\RequirePackage{fix-cm}

\documentclass{amsart}       

\usepackage{graphicx}
\usepackage{mathptmx}      
\usepackage{amsmath}              
\usepackage{amsopn,amssymb}       
\usepackage{enumerate}            
\usepackage{microtype}            
\usepackage{thmtools}             
\usepackage{mathtools}           

\newcommand{\bC}{\mathbb{C}}
\newcommand{\bN}{\mathbb{N}}
\newcommand{\bP}{\mathbb{P}}
\newcommand{\bR}{\mathbb{R}}
\newcommand{\bZ}{\mathbb{Z}}
\newcommand{\ind}{\!\!\uparrow}
\newcommand{\res}{\!\!\downarrow}
\DeclareMathOperator{\Ch}{Ch}
\DeclareMathOperator{\diag}{diag}
\DeclareMathOperator{\fin}{fin}
\DeclareMathOperator{\maj}{maj}
\DeclareMathOperator{\op}{op}
\DeclareMathOperator{\SYT}{SYT}

\theoremstyle{plain}
\newtheorem{theorem}{Theorem}
\newtheorem{lemma}{Lemma}
\newtheorem{proposition}{Proposition}
\newtheorem{corollary}{Corollary}

\newtheorem{conjecture}{Conjecture}
\newtheorem{question}{Question}

\theoremstyle{definition}

\newtheorem{definition}{definition}
\newtheorem{example}{Example}
\newtheorem{remark}{Remark}

\begin{document}

\title[Tableaux with Modular Major Index]{On the Existence of Tableaux with\\Given Modular Major Index}
\author{Joshua P. Swanson}
\address{Department of Mathematics, University of Washington,
Seattle, WA 98195, USA}
\urladdr{http://www.math.washington.edu/~jps314/}
\date{\today}

\begin{abstract}
  We provide simple necessary and sufficient conditions for the
  existence of a standard Young tableau of a given shape and major index $r$
  mod $n$, for all $r$. Our result generalizes the $r=1$ case due essentially to 
  Klyachko \cite{klyachko74} and proves a recent conjecture due to Sundaram
  \cite{sundaram16} for the $r=0$ case. A byproduct of the proof is an asymptotic 
  equidistribution result for ``almost all'' shapes. The proof uses a representation-theoretic
  formula involving Ramanujan sums and normalized symmetric group character estimates.
  Further estimates involving ``opposite'' hook lengths are given which are well-adapted to
  classifying which partitions $\lambda \vdash n$ have $f^\lambda \leq n^d$ for fixed $d$.
  We also give a new proof of a generalization of the hook length formula due to
  Fomin-Lulov \cite{fl95} for symmetric group characters at rectangles. We conclude with some
  remarks on unimodality of symmetric group characters.
\end{abstract}

\maketitle

\section{Introduction}
\label{sec:intro}

We assume basic familiarity with the combinatorics of Young tableaux and
the representation theory of the symmetric group. For further information and
definitions, see \cite{fulton97}, \cite{ec2}, or \cite{sagan01}.

Let $\lambda \vdash n$ be an integer partition of size $n$, and let $\SYT(\lambda)$ denote
the set of standard Young tableaux of shape $\lambda$. We write $\lambda'$ for the
transpose (or conjugate) of $\lambda$. Let $\maj T$ denote the major index
of $T \in \SYT(\lambda)$. We are chiefly interested in the counts
  \[ a_{\lambda, r} \coloneqq \#\{T \in \SYT(\lambda) : \maj T \equiv_n r\} \]
where $r$ is taken mod $n$. To avoid giving undue weight to trivial cases, we take $n \geq 1$
throughout. Work due to Klyachko and, later, Kra\'skiewicz--Weyman, gives the following.

\begin{theorem}[{\cite[Proposition 2]{klyachko74}}, {\cite{kw01}}]
  \label{thm:kly}
  Let $\lambda \vdash n$ and $n \geq 1$. The constant $a_{\lambda, 1}$ is positive except in
  the following cases, when it is zero:
  \begin{itemize}
    \item $\lambda = (2, 2)$ or $\lambda = (2, 2, 2)$;
    \item $\lambda = (n)$ when $n > 1$; or $\lambda = (1^n)$ when $n > 2$.
  \end{itemize}
\end{theorem}

Indeed, the counts $a_{\lambda, r}$ can be interpreted as irreducible multiplicities as
follows, a result originally due to Kra\'skiewicz--Weyman. Let $C_n$ be the cyclic group of
order $n$ generated by the long cycle $\sigma_n \coloneqq (1 2 \cdots n) \in S_n$, let
$S^\lambda$ be the
Specht module of shape $\lambda \vdash n$, and let $\chi^r \colon C_n \to \bC^\times$ be
the irreducible representation given by $\chi^r(\sigma_n^i) \coloneqq \omega_n^{ri}$ where
$\omega_n$ is a fixed primitive $n$th root of unity and $r \in \bZ/n$. Let $\langle -, -\rangle$
denote the standard scalar product for complex representations.

\begin{theorem}[see {\cite[Theorem~1]{kw01}}]
  \label{thm:mults}
  With the above notation, we have
    \[ \langle S^\lambda, \chi^r\ind_{C_n}^{S_n} \rangle
        = a_{\lambda, r} = \langle \chi^r, S^\lambda\res_{C_n}^{S_n} \rangle. \]
  Moreover, $a_{\lambda, r}$ depends only on $\lambda$ and $\gcd(n, r)$.
\end{theorem}

\begin{remark}
  Kra\'skiewicz-Weyman gave the first equality in Theorem \ref{thm:mults}, and the second
  follows by Frobenius reciprocity. Klyachko \cite[Proposition 2]{klyachko74} actually
  determined which $S^\lambda$ contain faithful representations of $C_n$ in agreement with
  Theorem \ref{thm:kly}. One may see through a variety of methods that
  $\chi^r\ind_{C_n}^{S_n}$ depends up to isomorphism only on $\gcd(r, n)$.
  
  The manuscript \cite{kw01} was long-unpublished, the delay being largely due to Klyachko
  having already given a significantly more direct proof of their main application, relating
  $\chi^1\ind_{C_n}^{S_n}$ to free Lie algebras, though we have no need of this connection.
  For a more modern and unified account of these results, see
  \cite[Theorems~8.8-8.12]{reutenauer93}.
\end{remark}

The following recent conjecture due to Sundaram was originally stated in terms of the
multiplicity of $S^\lambda$ in $1\ind_{C_n}^{S_n}$.

\begin{conjecture}{\cite{sundaram16}}.
  \label{conj:sundaram}
  Let $\lambda \vdash n$ and $n \geq 1$. Then $a_{\lambda, 0}$ is positive except in
  the following cases, when it is zero: $n > 1$ and
  \begin{itemize}
    \item $\lambda = (n-1, 1)$
    \item $\lambda = (2, 1^{n-2})$ when $n$ is odd
    \item $\lambda = (1^n)$ when $n$ is even.
  \end{itemize}
\end{conjecture}

\noindent Conjecture \ref{conj:sundaram} is the $r=0$ case of the following theorem, which
is our main result.

\begin{theorem}
  \label{thm:main}
  Let $\lambda \vdash n$ and $n \geq 1$. Then $a_{\lambda, r}$ is positive except in
  the following cases, when it is zero: $n > 1$ and
  \begin{itemize}
    \item $\lambda = (2, 2)$, $r = 1, 3$; or $\lambda = (2, 2, 2)$, $r = 1, 5$;
      or $\lambda = (3, 3)$, $r = 2, 4$;
    \item $\lambda = (n-1, 1)$ and $r=0$;
    \item $\lambda = (2, 1^{n-2})$,
      $r = \begin{cases}
                0 & \text{if $n$ is odd} \\
                \frac{n}{2} & \text{if $n$ is even};
              \end{cases}$
    \item $\lambda = (n)$, $r \in \{1, \ldots, n-1\}$;
    \item $\lambda = (1^n)$,
      $r \in \begin{cases}
                  \{1, \ldots, n-1\} & \text{if $n$ is odd} \\
                  \{0, \ldots, n-1\} - \{\frac{n}{2}\} & \text{if $n$ is even}.
                \end{cases}$
  \end{itemize}
  Equivalently, using Theorem \ref{thm:mults}, every irreducible representation appears in each
  $\chi^r\ind_{C_n}^{S_n}$ or $S^\lambda\res_{C_n}^{S_n}$ except in the noted exceptional
  cases.
\end{theorem}

M.~Johnson \cite{MR2276964} gave an alternative proof of Klyachko's result, Theorem
\ref{thm:kly}, involving explicit constructions with standard tableaux.
Kov\'acs--St\"ohr \cite{MR2228926} gave a different proof using the Littlewood--Richardson
rule which also showed that $a_{\lambda, 1} > 1$ implies $a_{\lambda, 1}
\geq \frac{n}{6} - 1$. Our approach is instead based on normalized symmetric group character
estimates. It has the benefit of yielding both more general and vastly more precise
estimates for $a_{\lambda, r}$.

Our starting point is the following character formula. See
Section \ref{sec:formula} for further discussion of its origins and a generalization.
Let $\chi^\lambda(\mu)$ denote the character of $S^\lambda$ at a permutation of cycle type
$\mu$. We write $\ell^{n/\ell}$ for the rectangular partition $(\ell, \ldots, \ell)$ with $\ell$
columns and $n/\ell$ rows. Write $f^\lambda \coloneqq \chi^\lambda(1^n) =
\dim S^\lambda = \#\SYT(\lambda)$.

\begin{theorem}
  \label{thm:formula}
  Let $\lambda \vdash n$ and $n \geq 1$. For all $r \in \bZ/n$,
    \[ \frac{a_{\lambda,r}}{f^\lambda}
        = \frac{1}{n} + \frac{1}{n} \sum_{\substack{\ell \mid n \\ \ell \neq 1}}
           \frac{\chi^\lambda(\ell^{n/\ell})}{f^\lambda} c_\ell(r) \]
  where
    \[ c_\ell(r) \coloneqq \mu\left(\frac{\ell}{\gcd(\ell, r)}\right)
        \frac{\phi(\ell)}{\phi(\ell/\gcd(\ell, r))} \]
  is a Ramanujan sum, $\mu$ is the classical M\"obius function, and
  $\phi$ is Euler's totient function.
\end{theorem}

We estimate the quotients in the preceding formula using the following result due to Fomin and
Lulov.

\begin{theorem}{\cite[Theorem~1.1]{fl95}}
  \label{thm:fl_bound}
  Let $\lambda \vdash n$ where $n = \ell s$. Then
    \[ |\chi^\lambda(\ell^s)| \leq \frac{s! \ell^s}{(n!)^{1/\ell}} (f^\lambda)^{1/\ell}. \]
\end{theorem}

The character formula in Theorem \ref{thm:formula} and the Fomin-Lulov bound are combined
below to give the following asymptotic uniform distribution result.

\begin{theorem}
  \label{thm:asymptotic}
  For all $\lambda \vdash n \geq 1$ and all $r$,
  \begin{equation}
    \label{eq:asymptotic}
    \left|\frac{a_{\lambda,r}}{f^\lambda} - \frac{1}{n}\right|
        \leq \frac{2n^{3/2}}{\sqrt{f^\lambda}}.
  \end{equation}
\end{theorem}

In Section \ref{sec:thm_proofs} we use ``opposite hook lengths'' to give a lower bound for
$f^\lambda$, Corollary \ref{cor:h_op}. These bounds, together with a somewhat more careful
analysis involving the character formula, Stirling's approximation, and the Fomin-Lulov bound,
are used to deduce both our main result, Theorem \ref{thm:main}, and the following more
explicit uniform distribution result.

\begin{theorem}
  \label{thm:dist}
  Let $\lambda \vdash n$ be a partition where $f^\lambda \geq n^5 \geq 1$.
  Then for all $r$,
    \[ \left|\frac{a_{\lambda, r}}{f^\lambda} - \frac{1}{n}\right| < \frac{1}{n^2}. \]
  In particular, if $n \geq 81$, $\lambda_1 < n-7$, and $\lambda_1' < n-7$, then
  $f^\lambda \geq n^5$ and the inequality holds.
\end{theorem}

Indeed, the upper bound in Theorem \ref{thm:dist} is quite weak and is intended only to
convey the flavor of the distribution of $(a_{\lambda, r})_{r=0}^{n-1}$ for fixed $\lambda$.
One may use Roichman's asymptotic estimate \cite{roichman96} of
$|\chi^\lambda(\ell^s)|/f^\lambda$ to prove exponential decay in many cases. Moreover, one
typically expects $f^\lambda$ to grow super-exponentially, i.e.~like $(n!)^{\epsilon}$ for some
$\epsilon > 0$ (see \cite{ls08} for some discussion and a more recent generalization of
Roichman's result), which in turn would give a super-exponential decay rate in Theorem
\ref{thm:dist}. We have no need for such explicit, refined statements and so have not pursued
them further.

Theorem \ref{thm:fl_bound} is based on the following generalization of the hook length
formula (the $\ell=1$ case), which seems less well-known than it deserves. We give an
alternate proof of Theorem \ref{thm:hook} in Section \ref{sec:hook} along with further
discussion. A \textit{ribbon} is a connected skew shape with no $2 \times 2$ rectangles.
For $\lambda \vdash n$, write $c \in \lambda$ to mean that $c$ is a cell in
$\lambda$. Further write $h_c$ for the \textit{hook length} of $c$ and write
$[n] \coloneqq \{1, 2, \ldots, n\}$.

\begin{theorem}[{\cite[2.7.32]{jk81}; see also \cite[Corollary~2.2]{fl95}}]
  \label{thm:hook}
  Let $\lambda \vdash n$ where $n = \ell s$. Then
  \begin{equation}
    \label{eq:hook_prod}
    |\chi^\lambda(\ell^s)|
        = \frac{\prod\limits_{\substack{i \in [n] \\ i \equiv_\ell 0}} i}
                   {\prod\limits_{\substack{c \in \lambda \\ h_c \equiv_\ell 0}} h_c}
  \end{equation}
  whenever $\lambda$ can be written as $s$ successive ribbons of length
  $\ell$ (i.e.~whenever the $\ell$-core of $\lambda$ is empty), and $0$ otherwise.
\end{theorem}

Other work on $q$-analogues of the hook length formula has focused on algebraic
generalizations and variations on the hook walk algorithm rather than evaluations of symmetric
group characters. For instance, an application of Kerov's $q$-analogue of the hook walk
algorithm \cite{kerov93} was to prove a recursive characterization of the right-hand side of
\eqref{eq:hook_q} below. See \cite[\S6]{cfkp11} for a relatively recent overview of literature
in this direction.

The rest of the paper is organized as follows. In Section \ref{sec:background}, we recall earlier
work. In Section \ref{sec:formula} we discuss and generalize Theorem \ref{thm:formula}. In
Section \ref{sec:thm_proofs}, we use symmetric group character estimates and a new estimate
involving ``opposite hook products,'' Proposition \ref{prop:h_op1}, to deduce our main results,
Theorem \ref{thm:main} and Theorem \ref{thm:dist}. We give an alternative proof of
Theorem \ref{thm:hook} in Section \ref{sec:hook}. In Section \ref{sec:unimodality},
we briefly discuss unimodality of symmetric group characters in light of Proposition
\ref{prop:h_op1}.

\section{Background}
\label{sec:background}

Here we review objects famously studied by Springer \cite[(4.5)]{springer74} and
Stembridge \cite{stembridge89} and give further background for use in later sections.
All representations will be finite-dimensional over $\bC$.

Continuing our earlier notation, $\lambda \vdash n$ is a partition of size $n$, $\SYT(\lambda)$
is the set of standard Young tableaux of shape $\lambda$ and which has cardinality
$f^\lambda$, $(12\cdots n)$ is the long cycle in the symmetric group $S_n$, $S^\lambda$ is
the irreducible $S_n$-module (Specht module) of shape $\lambda$ with character at an
element of cycle type $\mu$ given by $\chi^\lambda(\mu)$, $c \in \lambda$ denotes a cell in
the Ferrers diagram of $\lambda$, and $h_c$ denotes the hook length of that cell.

Let $G$ be a finite group, $g \in G$ a fixed element of order $n$, $M$ a finite dimensional
$G$-module, and $\omega_n$ a fixed primitive $n$th root of unity. Suppose
$\{\omega_n^{e_1}, \omega_n^{e_2}, \ldots\}$ is the multiset of eigenvalues
of $g$ acting on $M$. The multiset $\{e_1, e_2, \ldots\}$ lists the
\textit{cyclic exponents} of $g$ on $M$; these integers are well-defined
mod $n$. Following \cite{stembridge89}, define the corresponding ``modular'' generating
function as
  \[ P_{M, g}(q) \coloneqq q^{e_1} + q^{e_2} + \cdots \qquad \text{(mod $(q^n - 1)$)}. \]
Write $\chi^M(g)$ to denote the character of $M$ at $g$. Note that
\begin{equation}
  P_{M, g}(\omega_n^s) = \chi^M(g^s),
\end{equation}
so that for instance $P_{M, g}(q)$ depends only on the conjugacy class of $g$. When
$G = S_n$ and $g \in S_n$ has cycle type $\mu \vdash n$, we write
$P_{M, \mu}(q) \coloneqq P_{M, g}(q)$.

\begin{theorem}[{see \cite[Theorem~3.3]{stembridge89}} and \cite{kw01}]
  \label{thm:cyclic_exps}
  Let $\lambda \vdash n$. The cyclic exponents of $(1 2 \cdots n)$ on $S^\lambda$ are the
  major indices of $\SYT(\lambda)$, mod $n$, and
  \begin{align}
    \label{eq:cyc_exps}
    \begin{split}
      P_{S^\lambda, (n)}(q)
        &\equiv \sum_{T \in \SYT(\lambda)} q^{\maj T} \\
        &\equiv \sum_{r \mid n} a_{\lambda,r}
            \left(\sum_{\substack{1 \leq i \leq n \\ \gcd(i, n) = r}} q^i\right)
            \qquad \text{(mod $(q^n - 1)$)}.
    \end{split}
  \end{align}
\end{theorem}

\begin{remark}
  Stembridge gave the first equality in Theorem \ref{thm:cyclic_exps}. Equality of the first and
  third terms follows immediately from Kra\'skiewicz-Weyman's work using Theorem
  \ref{thm:mults} and the observation that the multiplicity of $\chi^r$ in
  $S^\lambda\res_{C_n}^{S_n}$ is the number of times $r$ appears as a cyclic exponent of
  $(1 2 \cdots n)$ in $S^\lambda$.
\end{remark}

We also recall Stanley's $q$-analogue of the hook length formula.

\begin{theorem}{\cite[7.21.5]{ec2}}
  \label{thm:q_hook}
  Let $\lambda \vdash n$ with $\lambda = (\lambda_1, \lambda_2, \ldots)$. Then
  \begin{equation}
    \label{eq:hook_q}
    \sum_{T \in \SYT(\lambda)} q^{\maj(T)}
       = \frac{q^{d(\lambda)}[n]_q!}
                  {\prod_{c \in \lambda} [h_c]_q}
  \end{equation}
  where $d(\lambda) \coloneqq \sum (i-1)\lambda_i$, $[n]_q! \coloneqq [n]_q [n-1]_q \cdots [1]_q$, and
  $[a]_q \coloneqq 1 + q + \cdots + q^{a-1} = \frac{q^a - 1}{q - 1}$.
\end{theorem}

The representation-theoretic interpretation of the coefficients $a_{\lambda, r}$ in
Theorem \ref{thm:mults} is related to the following result due independently to Lusztig
(unpublished) and Stanley. We record it to give our results context, though it will
not be used in our present work. For $\lambda \vdash n$ and $i \in \bZ$, define
  \[ b_{\lambda, i} \coloneqq \#\{T \in \SYT(\lambda) : \maj T = i\} \]
so that $\sum_{k \in \bZ} b_{\lambda, i+kn} = a_{\lambda, i}$.

\begin{theorem}{\cite[Proposition~4.11]{stanley79}}
  \label{thm:coinv}
  Let $\lambda \vdash n$. The multiplicity of $S^\lambda$ in the $i$th graded piece of the
  type $A_{n-1}$ coinvariant algebra is $b_{\lambda, i}$.
\end{theorem}
\noindent Indeed, the second equality in Theorem \ref{thm:mults} follows from
Theorem \ref{thm:coinv} and \cite[Prop.~4.5]{springer74}. See also \cite[p.~3059]{abr05} for
a more recent refinement of Theorem \ref{thm:coinv} and some further discussion.

Finally, we have need of the so-called Ramanujan sums.
\begin{definition}
  Given $j \in \bZ_{> 0}$ and $s \in \bZ$, the corresponding Ramanujan sum is
  \begin{align*}
    c_j(s)
      &\coloneqq \text{the sum of the $s$th powers of the primitive $j$th roots of unity.}
  \end{align*}
\end{definition}
For instance, $c_4(2) = i^2 + (-i)^2 = -2 = \mu(4/2) \phi(4)/\phi(2)$.
The equivalence of this definition of $c_j(s)$ and the formula in Theorem \ref{thm:formula} is
classical and was first given by H\"older; see \cite[Lemma~7.2.5]{knopfmacher75} for a more
modern account. These sums satisfy the well-known relation
\begin{equation}
  \label{eq:ram_sums_id}
  \sum_{v \mid n} c_v(n/s) c_r(n/v)
      = \begin{cases}
          n & r=s \\
          0 & r \neq s
         \end{cases}
\end{equation}
for all $s, r \mid n$ \cite[Lemma~7.2.2]{knopfmacher75}.

\section{Generalizing the Character Formula}
\label{sec:formula}

In this section we discuss Theorem \ref{thm:formula} and present a straightforward
generalization. We begin with a proof of Theorem \ref{thm:formula} similar to but different
from that in \cite{desarmenien90}. It is included chiefly because of its simplicity given the
background in Section \ref{sec:background} and because part of the argument will be used
below in Section \ref{sec:hook}.

\begin{proof}[of Theorem \ref{thm:formula}]

Pick $s \mid n$, so $(12\cdots n)^s$ has cycle type
$((n/s)^s)$. Evaluating \eqref{eq:cyc_exps} at $q=\omega_n^s$ gives
\begin{align}
  \label{eq:lin_comb1}
  \begin{split}
    \chi^\lambda((n/s)^s) = P_{S^\lambda,(n)}(\omega_n^s)
      &= \sum_{r \mid n} a_{\lambda, r} c_{n/r}(s)
  \end{split}
\end{align}
since $(\omega_n^s)^i = (\omega_n^i)^s$ and $\omega_n^i$ is a primitive
$n/\gcd(i, n)$th root of unity. Equation \eqref{eq:lin_comb1} gives a system of linear
equations, one for each $s$ such that $s \mid n$, and with variables $a_{\lambda, r}$
for each $r \mid n$. The coefficient matrix is
$C \coloneqq (c_{n/r}(s))_{s \mid n, r \mid n}$. For example, the
$s=n$ linear equation reads
  \[ f^\lambda = \chi^\lambda(1^n) = \sum_{r \mid n} a_{\lambda, r} \phi(n/r), \]
which follows immediately from the fact that $f^\lambda = \sum_{r=0}^{n-1} a_{\lambda, r}$
and that $a_{\lambda, r}$ depends only on $\gcd(r, n)$. 

As it happens, the coefficient matrix $C$ is nearly its own inverse. Precisely,
\begin{equation}
  \label{eq:square_inverse}
  (c_{n/r}(s))_{s \mid n, r \mid n}^2 = n\,I,
\end{equation}
where $I$ is the identity matrix with as many rows as positive divisors of $n$.
It is easy to see that \eqref{eq:square_inverse} is equivalent to the identity
\eqref{eq:ram_sums_id} above. Using \eqref{eq:square_inverse} to invert
\eqref{eq:lin_comb1} gives
  \[ a_{\lambda, r}n = \sum_{s \mid n} \chi^\lambda((n/s)^s) c_{n/s}(r). \]
For the $s=n$ term, we have $c_1(r) = 1$ and $\chi^\lambda(1^n) = f^\lambda$. Tracking
this term separately, dividing by $n$ and replacing $s$ with $\ell \coloneqq n/s$ now gives
Theorem \ref{thm:formula}, completing the proof.
\qed\end{proof}

Variations on Theorem \ref{thm:formula} have appeared in the literature numerous times in
several guises, sometimes implicitly (see \cite[Th\'eor\`eme~2.2]{desarmenien90},
\cite[(7)]{klyachko74}, or \cite[7.88(a), p.~541]{ec2}). In this section we write out a precise 
and relatively general version of these results which explicitly connects Theorem
\ref{thm:formula} to the well-known corresponding symmetric function expansion due to
H.~O.~Foulkes. Let $\Ch$ denote the Frobenius characteristic map and let $p_\lambda$
denote the power symmetric function indexed by the partition $\lambda$.

\begin{theorem}{\cite[Theorem 1]{foulkes72}}
  \label{thm:chir_p}
  Suppose $\lambda \vdash n \geq 1$ and $r \in \bZ/n$. Then
  \begin{equation}
    \label{eq:chir_p}
    \Ch \chi^r\ind_{C_n}^{S_n} = \frac{1}{n} \sum_{\ell \mid n} c_\ell(r) p_{(\ell^{n/\ell})}.
  \end{equation}
\end{theorem}

The following straightforward result, essentially implicit in \cite[7.88(a), p.~541]{ec2},
connects and generalizes Theorem \ref{thm:chir_p} and Theorem \ref{thm:formula}.

\begin{theorem}
  \label{thm:char_formulas}
  Let $H$ be a subgroup of $S_n$ and let $M$ be a finite-dimensional $H$-module with
  character $\chi^M \colon H \to \bC$. Then
  \begin{equation}
    \label{eq:M_p}
    \Ch M\ind_H^{S_n} = \frac{1}{|H|} \sum_{\mu \vdash n} c_\mu p_\mu
  \end{equation}
  and, for all $\lambda \vdash n$,
  \begin{equation}
    \label{eq:M_mult}
    \langle M\ind_H^{S_n}, S^\lambda\rangle = \frac{1}{|H|} \sum_{\mu \vdash n}
        c_\mu \chi^\lambda(\mu),
  \end{equation}
  where
    \[ c_\mu \coloneqq \sum_{\substack{h \in H\\\tau(h)=\mu}} \chi^M(h) \]
  and $\tau(\sigma)$ denotes the cycle type of the permutation $\sigma$.
\end{theorem}
  \begin{proof}
    Write $N \coloneqq M\ind_H^{S_n}$. By definition (see \cite[p.~351]{ec2}),
    \begin{equation}
      \label{eq:N_p}
      \Ch N = \sum_{\mu \vdash n} \frac{\chi^N(\mu)}{z_\mu} p_\mu
    \end{equation}
    where $z_\mu$ is the order of the stabilizer of any permutation of cycle type $\mu$
    under conjugation. From the induced character formula (see \cite[7.2, Prop.~20]{serre77}),
    we have
      \[ \chi^N(\sigma) = \frac{1}{|H|}
            \sum_{\substack{a \in S_n\\\text{s.t. }a \sigma a^{-1} \in H}}
              \chi^M(a \sigma a^{-1}). \]
    Say $\tau(\sigma) = \mu$. Each $a\sigma a^{-1} = h \in H$ with $\tau(h) = \mu$ appears
    in the preceding sum $z_\mu$ times, since $\sigma$ and $h$ are conjugate and
    $z_\mu$ is also the number of ways to conjugate any fixed permutation with cycle type
    $\mu$ to any other fixed permutation with cycle type $\mu$. Hence
    \begin{equation}
      \label{eq:N_sigma}
      \chi^N(\mu) = \frac{1}{|H|} \sum_{\substack{h \in H\\\tau(h)=\mu}}
            z_\mu \chi^M(h).
    \end{equation}
    Equation \eqref{eq:M_p} now follows from \eqref{eq:N_p} and \eqref{eq:N_sigma}.
    Equation \eqref{eq:M_mult} follows from \eqref{eq:M_p} in the usual way using the fact
    (see \cite[(7.76)]{ec2}) that $p_\mu = \sum_\lambda \chi^\lambda(\mu) s_\lambda$.
  \qed\end{proof}

Note that \eqref{eq:M_p} specializes to Theorem \ref{thm:chir_p} and \eqref{eq:M_mult}
specializes to Theorem \ref{thm:formula} when $M = \chi^r$. In that case, the only possibly
non-zero $c_\mu$ arise from $\mu = (\ell^{n/\ell})$ for $\ell \mid n$.

One may consider analogues of the counts $a_{\lambda, r}$ obtained by inducing other
one-dimensional representations of subgroups of $S_n$. Motivated by the study of so-called
higher Lie modules, there is a natural embedding of reflection groups $C_a \wr S_b
\hookrightarrow S_{ab}$. A classification analogous to Klyachko's result, Theorem
\ref{thm:kly}, was asserted for $b=2$ by Schocker \cite[Theorem 3.4]{schocker03},
though the ``rather lengthy proof'' making ``extensive use of routine applications of the
Littlewood-Richardson rule and some well-known results from the theory of plethysms'' was
omitted. By contrast, our approach using Theorem \ref{thm:char_formulas} may be pushed
through in this case using an appropriate generalization of the Fomin-Lulov
bound, such as \cite[Theorem 1.1]{ls08}, resulting in analogues
of Theorem \ref{thm:main} and Theorem \ref{thm:dist}. Our approach begins to break down
when $b$ is large relative to $n=ab$ and \eqref{eq:M_mult} has many terms. However, we
have no current need for such generalizations and so have not pursued them further.

\section{Proof of the Main Results}
\label{sec:thm_proofs}

We now turn to the proofs of Theorem \ref{thm:main}, Theorem \ref{thm:asymptotic},
and Theorem \ref{thm:dist}. We begin by combining the Fomin--Lulov bound and Stirling's
approximation, which quickly gives Theorem \ref{thm:asymptotic}. We then use somewhat
more careful estimates to give a sufficient condition, $f^\lambda \geq n^3$, for $a_{\lambda, r} \neq 0$. Afterwards we give an inequality between hook length products
and ``opposite'' hook length products, Proposition \ref{prop:h_op1}, from which we classify
$\lambda$ for which $f^\lambda < n^3$. Theorem \ref{thm:main} follows in
almost all cases, with the remainder being handled by brute force computer verification and
case-by-case analysis. Theorem \ref{thm:dist} will be similar, except the bound
$f^\lambda < n^5$ will be used.

\begin{lemma}
  \label{lem:FL_bound}
  Suppose $\lambda \vdash n = \ell s$. Then
  \begin{equation}
    \label{eq:fl_bound}
    \ln \frac{|\chi^\lambda(\ell^s)|}{f^\lambda}
      \leq \left(1-\frac{1}{\ell}\right) \left[\frac{1}{2} \ln n - \ln f^\lambda
            + \ln \sqrt{2\pi}\right] + \frac{\ell}{12n} - \frac{1}{2} \ln \ell.
  \end{equation}
\end{lemma}
  \begin{proof}
    We apply the following version of Stirling's approximation \cite[(1.53)]{spencer14}.
    For all $m \in \bZ_{>0}$,
      \[ \left(m + \frac{1}{2}\right) \ln m - m + \ln \sqrt{2\pi}
          \leq \ln m! \leq \left(m + \frac{1}{2}\right) \ln m - m + \ln \sqrt{2\pi} +
          \frac{1}{12m}.
      \]
    The Fomin--Lulov bound, Theorem \ref{thm:fl_bound}, gives
    \begin{align*}
      \frac{|\chi^\lambda(\ell^s)|}{f^\lambda}
         \leq \frac{\frac{n}{\ell}! \ell^{n/\ell}}{(n!)^{1/\ell} (f^\lambda)^{1 - 1/\ell}}.
    \end{align*}
    Combining these gives
    \begin{align*}
      \ln \frac{|\chi^\lambda(\ell^s)|}{f^\lambda}
        &\leq \ln \left(\frac{n}{\ell}\right)! + \frac{n}{\ell} \ln \ell - \frac{1}{\ell} \ln n!
             - \left(1-\frac{1}{\ell}\right) \ln f^\lambda \\
        &\leq \left(\frac{n}{\ell} + \frac{1}{2}\right) \ln \frac{n}{\ell} - \frac{n}{\ell}
             + \ln \sqrt{2\pi} + \frac{\ell}{12 n} + \frac{n}{\ell} \ln \ell \\
        &\quad - \frac{1}{\ell}\left(\left(n + \frac{1}{2}\right) \ln n - n + \ln \sqrt{2\pi}\right)
             - \left(1 - \frac{1}{\ell}\right) \ln f^\lambda \\
        &= \frac{1}{2} \ln \frac{n}{\ell} + \ln \sqrt{2\pi} + \frac{\ell}{12n}
             - \frac{1}{2\ell} \ln n - \frac{\ln \sqrt{2\pi}}{\ell}
             - \left(1 - \frac{1}{\ell}\right) \ln f^\lambda.
    \end{align*}
    Rearranging this final expression gives \eqref{eq:fl_bound}.
  \qed\end{proof}

We may now prove Theorem \ref{thm:asymptotic}.

\begin{proof}[of Theorem \ref{thm:asymptotic}]
  For $\ell \leq 2 \leq n$, applying simple term-by-term estimates to \eqref{eq:fl_bound}
  gives
    \[ \ln\frac{|\chi^\lambda(\ell^s)|}{f^\lambda}
        \leq \frac{1}{2} \ln n - \frac{1}{2} \ln f^\lambda + \ln \sqrt{2\pi} + \frac{1}{12}
               - \frac{\ln 2}{2}. \]
  Consequently,
    \[ \frac{|\chi^\lambda(\ell^s)|}{f^\lambda} \leq C \sqrt{\frac{n}{f^\lambda}} \]
  where $C = \sqrt{\pi} \exp(1/12) \approx 1.93 < 2$. The Ramanujan sums
  $c_\ell(r)$ have the trivial bound $|c_\ell(r)| \leq \ell \leq n$. The estimate in
  Theorem \ref{thm:asymptotic} now follows immediately from Theorem \ref{thm:formula}.
\qed\end{proof}

\begin{lemma}
  \label{lem:phi_d}
  Pick $\lambda \vdash n$ and $d \in \bR$. Suppose for all $1 \neq \ell \mid n$ where
  $\lambda$ may be written as $s \coloneqq n/\ell$ successive ribbons each of length $\ell$ that
  \begin{equation}
    \label{eq:ratio_ub}
    \frac{|\chi^\lambda(\ell^s)|}{f^\lambda} \leq \frac{1}{n^d \phi(\ell)}.
  \end{equation}
  Then for all $r \in \bZ/n$,
    \[ \left|\frac{a_{\lambda, r}}{f^\lambda} - \frac{1}{n}\right| < \frac{1}{n^d}. \]
\end{lemma}
  \begin{proof}
    By Theorem \ref{thm:formula}, we must show
      \[ \frac{1}{n} \left|\sum_{\substack{\ell \mid n \\ \ell \neq 1}}
          \frac{\chi^\lambda(\ell^s)}{f^\lambda} c_\ell(r)\right|
          < \frac{1}{n^d}. \]
    Using the explicit form for $c_\ell(r)$ in Theorem \ref{thm:formula} and the
    fact that $n$ has fewer than $n$ proper divisors, it suffices to show
      \[ \left|\frac{\chi^\lambda(\ell^s)}{f^\lambda} \phi(\ell)\right|
          \leq \frac{1}{n^d} \]
    for all $\ell \mid n$, $\ell \neq 1$, so the result follows from our assumption
    \eqref{eq:ratio_ub}.
  \qed\end{proof}

\begin{corollary}
  \label{cor:n_cubed}
  Let $\lambda \vdash n$. If $f^\lambda \geq n^3 \geq 1$, then
  $a_{\lambda,r} \neq 0$.
\end{corollary}
  \begin{proof}
    Equation \eqref{eq:fl_bound} gives
    \begin{align}
      \label{eq:char_ratio}
        \ln \frac{|\chi^\lambda(\ell^s)|}{f^\lambda}
          &\leq \left(1-\frac{1}{\ell}\right) \left[-\frac{5}{2} \ln n + \ln \sqrt{2\pi}\right]
             + \frac{\ell}{12n} - \frac{1}{2} \ln \ell.
    \end{align}
    At $\ell=2$, the right-hand side of \eqref{eq:char_ratio} is less than $\ln\frac{1}{\phi(2)n}$
    for $n \geq 3$. At $\ell=3, 4, 5$, the same expression is less than
    $\ln\frac{1}{\phi(\ell) n}$ for $n \geq 4, 3, 5$, respectively. At $\ell \geq 6$, applying
    simple term-by-term estimates to \eqref{eq:char_ratio} gives
    \begin{equation}
      \label{eq:char_ratio2}
      \ln \frac{|\chi^\lambda(\ell^s)|}{f^\lambda} \leq
         -\left(1 - \frac{1}{6}\right) \frac{5}{2} \ln n
         + \ln \sqrt{2\pi} + \frac{1}{12} - \frac{1}{2} \ln 6
    \end{equation}
    which is less than $\ln\frac{1}{n^2}$ for $n \geq 4$. Thus, 
    Lemma \ref{lem:phi_d} applies with $d=1$ for all $n \geq 5$, so that
      \[ \left|\frac{a_{\lambda, r}}{f^\lambda} - \frac{1}{n}\right| < \frac{1}{n}, \]
    and in particular $a_{\lambda, r} \neq 0$. The cases $1 \leq n \leq 4$ remain,
    but they may be easily checked by hand.
  \qed\end{proof}

We next give techniques that are well-adapted to classifying $\lambda \vdash n$
for which $f^\lambda < n^d$ for fixed $d$. We begin with a curious observation,
Proposition \ref{prop:h_op1}, which is similar in flavor to \cite[Theorem~2.3]{fl95}.
It was also recently discovered independently by Morales--Panova--Pak as a
corollary of the Naruse hook length formula for skew shapes; see
\cite[Proposition~12.1]{mpp17}. See also \cite{243846} for further discussion and an
alternate proof of a stronger result by F.~Petrov.

\begin{definition}
  Consider a partition $\lambda = (\lambda_1, \ldots, \lambda_m)$ with
  $\lambda_1 \geq \lambda_2 \geq \cdots \geq 0$ as a set of cells (in French notation)
    \[ \lambda = \{(a, b) \in \bZ \times \bZ : 1 \leq b \leq m, 1 \leq a \leq \lambda_b\}. \]
  Given a cell $c = (a, b) \in \lambda \subset \bN \times \bN$, the
  \textit{opposite hook length} $h_c^{\op}$ at $c$ is $a+b-1$. For instance,
  the unique cell in $\lambda = (1)$ has opposite hook length $1$, and the opposite hook
  length increases by $1$ for each north or east step.
\end{definition}

It is easy to see that $\sum_{c \in \lambda} h_c^{\op} = \sum_{c \in \lambda} h_c$.
On the other hand, we have the following inequality for their products.

\begin{proposition}
  \label{prop:h_op1}
  For all partitions $\lambda$,
    \[ \prod_{c \in \lambda} h_c^{\op} \geq \prod_{c \in \lambda} h_c. \]
  Moreover, equality holds if and only if $\lambda$ is a rectangle.
\end{proposition}
  \begin{proof}
    If $\lambda$ is a rectangle, the multisets $\{h_c^{\op}\}$ and $\{h_c\}$ are equal,
    so the products agree. The converse will be established in the course of proving the
    inequality. For that, we begin with a simple lemma.
    \begin{lemma}
      \label{lem:prod_switch}
      Let $x_1 \geq \cdots \geq x_m \geq 0$ and
      $y_1 \geq \cdots \geq y_m \geq 0$ be real numbers. Then
        \[ \prod_{i=1}^m (x_i + y_i) \leq \prod_{i=1}^m (x_i + y_{m-i+1}). \]
      Moreover, equality holds if and only if for all $i$ either $x_i = x_{m-i+1}$ or
      $y_i = y_{m-i+1}$.
    \end{lemma}
      \begin{proof}
        If $m=1$, the result is trivial. If $m=2$, we compute
          \[ (x_1 + y_2)(x_2 + y_1) - (x_1 + y_1)(x_2 + y_2)
              =  (x_1 - x_2) (y_1 - y_2) \geq 0. \]
        The result follows in general by pairing terms $i$ and $m-i+1$ and using these base cases.
      \qed\end{proof}
    
    \begin{figure}[b]
      \centering
      \includegraphics[scale=1]{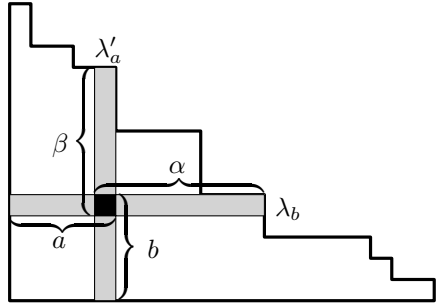}
      \caption{Arm length $\alpha$, co-arm length $a$, leg length $\beta$, co-leg length $b$
        for $c = (a, b) \in \lambda$. The hook length is $h_c = \alpha+\beta-1$ and the opposite
        hook length is $h_c^{\op} = a+b-1$}
      \label{fig:hooks}
    \end{figure}
    
    Returning to the proof of the proposition, the strategy will be to break up $h_c$ and
    $h_c^{\op}$ in terms of (co-)arm and (co-)leg lengths, and apply the lemma to each
    column of $\lambda$ when computing $\prod h_c$, or equivalently to each row of
    $\lambda$ when computing $\prod h_c^{\op}$. More precisely, let
    $c = (a, b) \in \lambda$. Take $\lambda = (\lambda_1, \lambda_2, \ldots)$ and
    $\lambda' = (\lambda_1', \lambda_2', \ldots)$. Define the \textit{co-arm length} of $c$
    as $a$, the \textit{co-leg length} of $c$ as $b$, the \textit{arm length} of $c$ as
    $\alpha \coloneqq \alpha(a, b) \coloneqq \lambda_b - a + 1$, and the \textit{leg length} of $c$ as
    $\beta \coloneqq \beta(a, b) \coloneqq \lambda_a' - b + 1$; see Figure \ref{fig:hooks}. With these
    definitions, we have $h_c^{\op} = a + b - 1$ and $h_c = \alpha + \beta - 1$. We now
    compute
    \begin{align*}
      \prod_{c \in \lambda} h_c^{\op}
        &= \prod_{(a, b) \in \lambda} (a+b-1)
          = \prod_b \prod_{a=1}^{\lambda_b} (a+b-1) \\
        &= \prod_b \prod_{a=1}^{\lambda_b} ((\lambda_b + 1 - a) + b - 1)
          = \prod_a \prod_{b=1}^{\lambda_a'} (\alpha + b - 1) \\
        &\geq \prod_a \prod_{b=1}^{\lambda_a'}
           (\alpha + (\lambda_a' + 1 - b) - 1) \\
        &= \prod_{(a, b) \in \lambda} (\alpha + \beta - 1)
         = \prod_{c \in \lambda} h_c,
    \end{align*}
    where Lemma \ref{lem:prod_switch} is used for the inequality with $i \coloneqq b$,
    $m \coloneqq \lambda_a'$, $x_i \coloneqq \alpha - 1 = \lambda_b - a$, $y_i \coloneqq \lambda_a' + 1 - b$. 
    Moreover, if equality occurs, then since the $y_i$ strictly decrease, we must have
    $\lambda_1 = \lambda_m$ for all $a$, forcing $\lambda$ to be a rectangle.
  \qed\end{proof}

It would be interesting to find a bijective explanation for Proposition \ref{prop:h_op1}.
The appearance of rectangles is particularly striking. Note,
however, that $n!/\prod_{c \in \lambda} h_c^{\op}$ need not be an integer. In any case, we
continue towards Theorem \ref{thm:main}.

\begin{definition}
  Define the \textit{diagonal preorder} on partitions as follows. Declare $\lambda 
  \lesssim^{\diag} \mu$ if and only if for all $i \in \bP$,
    \[ \#\{c \in \lambda : h_c^{\op} \geq i\} \leq \#\{d \in \mu : h_d^{\op} \geq i\}. \]
\end{definition}

Note that $\lesssim^{\diag}$ is reflexive and transitive, though not anti-symmetric, so the
diagonal preorder is not a partial order. For example, the partitions $(3, 1)$, $(2, 2)$, and
$(2, 1, 1)$ all have the same number of cells with each opposite hook length.
A straightforward consequence of the definition is that
\begin{equation}
  \label{eq:diag_preorder}
  \lambda \lesssim^{\diag} \mu \qquad \Rightarrow \qquad
      \prod_{c \in \lambda} h_c^{\op} \leq \prod_{d \in \mu} h_d^{\op}.
\end{equation}
Hooks are maximal elements of the diagonal preorder in a sense we next make precise.

\begin{definition}
  \label{def:diagonal_excess}
  Let $\lambda \vdash n$ for $n \geq 1$. The \textit{diagonal excess} of $\lambda$
  is
    \[ N(\lambda) \coloneqq |\lambda| - \max_{c \in \lambda} h_c^{\op}. \]
\end{definition}

\noindent For instance, $\lambda = (3, 3)$ has opposite hook lengths ranging from $1$ to 
$4$, so $N((3, 3)) = 6 - 4 = 2$.

The following simple observation will be used shortly.

\begin{proposition}
  \label{prop:diagonal_unimodality}
  Let $\lambda \vdash n$ for $n \geq 1$. Take $\pi \colon \lambda \to \bP$ via
  $\pi(c) \coloneqq h_c^{\op}$. Then the fiber sizes $|\pi^{-1}(i)|$ are unimodal, and are indeed of
  the form
    \[ 1 = |\pi^{-1}(1)| < \cdots < m = |\pi^{-1}(m)| \geq |\pi^{-1}(m+1)| \geq \cdots \]
  for some unique $m \geq 1$.
\end{proposition}
  \begin{proof}
    This follows quickly by considering the largest staircase shape
    contained in $\lambda$. Indeed, $m$ is the number of rows or columns in such a staircase.
  \qed\end{proof}

\begin{example}
  \label{ex:hook_diagonals}
  If $\lambda \vdash n$ is a hook, the sequence of fiber sizes in
  Proposition \ref{prop:diagonal_unimodality} is
    \[ 1 < 2 \geq 2 \geq 2 \cdots \geq 2 \geq 1 \geq \cdots \geq 1 \geq 0 \geq \cdots \]
  where there are $N(\lambda)$ two's and $n -  N(\lambda)$ non-zero entries. In particular, 
  $N(\lambda) + 1 \leq n - N(\lambda)$, i.e.~$2N(\lambda)+1 \leq n$.
\end{example}

\begin{proposition}
  \label{prop:h_op}
  Let $\lambda \vdash n$ for $n \geq 1$. Set 
  \begin{equation}
    \label{eq:N}
    N \coloneqq
        \begin{cases}
          N(\lambda) & \text{if }2N(\lambda)+1 \leq n \\
          \left\lfloor\frac{n-1}{2}\right\rfloor & \text{if }2N(\lambda)+1 > n.
        \end{cases}
  \end{equation}
  Then
  \begin{equation}
    \label{eq:diag_hook}
    \lambda \lesssim^{\diag} (n-N, 1^N).
  \end{equation}
  In particular, if $2N(\lambda) + 1 \leq n$, then the hook $(n - N(\lambda), 1^{N(\lambda)})$
  is maximal for the diagonal preorder on partitions of size $n$ with diagonal excess
  $N(\lambda)$.
\end{proposition}
  \begin{proof}
    Using Proposition \ref{prop:diagonal_unimodality}, the sequence
      \[ D(\lambda) \coloneqq \left(|\pi^{-1}(i)|\right)_{i \in \bP}. \]
    is of the form
      \[ D(\lambda) = (1, 2, \ldots, m, \ldots, 0, \ldots) \]
    where the terms weakly decrease starting at $m$. Given a sequence
    $D = (D_1, D_2, \ldots) \in \bN^{\bP}$, define $N(D) \coloneqq
    \sum_{i : D_i \neq 0} (D_i - 1)$. We have $N(D(\lambda)) = N(\lambda)$.
    Iteratively perform the following procedure starting with $D \coloneqq D(\lambda)$ as many
    times as possible; see Example \ref{ex:hook_reduction}.
    \begin{enumerate}[(i)]
      \item If $2N(D)+1 > n$ and some $D_i > 2$, choose $i$ maximal with this property.
        Decrease the $i$th entry of $D$ by $1$ and replace the first $0$ term in $D$ with $1$.
      \item If $2N(D)+1 \leq n$ and some $D_i > 2$, choose $i$ maximal with this property.
        We will shortly show that there is some $j>i$ for which $D_j = 1$. Choose $j$
        minimal with this property, decrease the $i$th term in $D$ by $1$, and increment the
        $j$th term by $1$.
    \end{enumerate}
    
    \begin{example}
      \label{ex:hook_reduction}
      Suppose $\lambda = (4, 4, 4, 4)$, so $n = 16$ and
        \[ D(\lambda) = (1, 2, 3, 4, 3, 2, 1, 0, \ldots), \]
      which we abbreviate as $D(\lambda) = 1234321$. Applying the procedure gives the
      following sequences, where modified entries are underlined:

      \begin{center}
        \begin{tabular}{l|c|c}
          $D$ & $N(D)$ & $2N(D) + 1$ \\
          \hline
          1234321 & 9 & 19 \\
          1234\underline{2}21\underline{1} & 8 & 17 \\
          123\underline{3}2211\underline{1} & 7 & 15 \\
          123\underline{2}22\underline{2}11 & 7 & 15 \\
          12\underline{2}2222\underline{2}1 & 7 & 15
      \end{tabular}
      \end{center}
    \end{example}
    
    Returning to the proof, for the claim in (ii), first note that both procedures preserve
    unimodality and the initial $1$ in $D(\lambda)$. Hence at any intermediate step, $D$ is of
    the form
      \[ (1, D_2, D_3, \ldots, D_k, 1, \ldots, 1, 0, \ldots) \]
    where $D_2, \ldots, D_k \geq 2$ and there are $\ell \geq 0$ terminal $1$'s. Since
    $2N(D) + 1 \leq n$, we have
    \begin{align*}
      2N(D) + 1
        &= 2(D_2 - 1 + \cdots + D_k - 1) + 1
        \leq n = 1 + D_2 + \cdots + D_k + \ell \\
        &\Leftrightarrow (D_2 - 2) + \cdots + (D_k - 2) \leq \ell,
    \end{align*}
    forcing $\ell > 0$ since by assumption some $D_i > 2$, giving the claim. The procedure
    evidently terminates.
    
    In applying (i), $N(D)$ decreases by $1$, whereas $N(D)$ is constant in applying (ii).
    For the final sequence $D_{\fin}$, it follows that $N(D_{\fin}) = N$ from \eqref{eq:N}.
    Both (i) and (ii) strictly increase in the natural diagonal partial order on sequences. The final 
    sequence will be
      \[ D_{\fin} = (1, 2, 2, \ldots, 2, 1, 1, \ldots, 1, 0, \ldots) \]
    where there are $N$ two's and $n-N$ non-zero entries. This is precisely $D((n-N, 1^N))$
    by Example \ref{ex:hook_diagonals}, and the result follows.
  \qed\end{proof}

We may now give a polynomial lower bound on $f^\lambda$.

\begin{corollary}
  \label{cor:h_op}
  Let $\lambda \vdash n$ for $n \geq 1$ and take $N$ as in \eqref{eq:N}. For any
  $0 \leq M \leq N$, we have
  \begin{equation}
    \label{eq:h_op_bound}
    \prod_{c \in \lambda} h_c^{\op} \leq (n-M)! (M+1)!.
  \end{equation}
  Moreover,
  \begin{equation}
    \label{eq:f_la_binom}
    f^\lambda \geq \frac{1}{M+1} \binom{n}{M}.
  \end{equation}
\end{corollary}
  \begin{proof}
    Equation \eqref{eq:h_op_bound} in the case $M=N$ follows by combining
    \eqref{eq:diag_preorder}
    and \eqref{eq:diag_hook}. The general case follows similarly upon noting $(n-N, 1^N)
    \lesssim^{\diag} (n-M, 1^M)$ since $N \leq \left\lfloor\frac{n-1}{2}\right\rfloor$.
    
    For \eqref{eq:f_la_binom}, use Proposition \ref{prop:h_op1} and
    \eqref{eq:h_op_bound} to compute
      \[ f^\lambda = \frac{n!}{\prod_{c \in \lambda} h_c}
          \geq \frac{n!}{\prod_{c \in \lambda} h_c^{\op}}
          \geq \frac{n!}{(n-M)!(M+1)!} = \frac{1}{M+1} \binom{n}{M}. \]
  \qed\end{proof}

We now prove Theorem \ref{thm:main} and Theorem \ref{thm:dist}.

\begin{proof}[of Theorem \ref{thm:main}]
  We begin by summarizing the verification of Theorem \ref{thm:main} for $n \leq 33$. For
  $1 \leq n \leq 33$, a computer check shows that one may use Corollary \ref{cor:n_cubed}
  for all but $688$ particular $\lambda$. However, the number of standard tableaux for these
  exceptional $\lambda$ is small enough that the conclusion of the theorem may be quickly 
  verified by computer. We now take $n \geq 34$.
  
  Let $N$ be as in \eqref{eq:N}. If $N \geq 5$, by Corollary \ref{cor:h_op},
    \[ f^\lambda \geq \frac{1}{6} \binom{n}{5} \geq n^3 \]
  for $n \geq 32$, so we may take $N \leq 4$. Since $\left\lfloor\frac{n-1}{2}\right\rfloor
  \geq 16 > 4 \geq N$, we must have $N = N(\lambda)$.
  
  Write $\nu \oplus \mu$ to denote the
  concatenation of partitions $\nu$ and $\mu$, where we assume the largest part of $\mu$
  is no larger than the smallest part of $\nu$. Using Proposition
  \ref{prop:diagonal_unimodality}, since $n \geq 32$ and $N=N(\lambda) \leq 4$, we find that
  either $\lambda = (n-N) \oplus \mu$ or $\lambda' = (n-N) \oplus \mu$ for $|\mu| = N$.
  
  To cut down on duplicate work, note that transposing $T \in \SYT(\lambda)$ complements
  the descent set of $T$. It follows that $b_{\lambda, i} =
  b_{\lambda', \binom{n}{2} - i}$, so that $a_{\lambda, r} = a_{\lambda',
  \binom{n}{2} - r}$. Since the statement of Theorem \ref{thm:main} also exhibits this
  symmetry, we may thus consider only the case when $\lambda = (n-N) \oplus \mu$.

  There are twelve $\mu$ with $|\mu| \leq 4$. One may check that the five possible $\mu$ for
  $N = 4$ all result in $f^\lambda \geq n^3$ for $n \geq 34$, leaving seven remaining $\mu$,
  namely $\mu = \varnothing, (1), (2), (1, 1), (3), (2, 1), (1, 1, 1)$. It is straightforward
  though tedious to verify the conclusion of Theorem \ref{thm:main} in each of these cases. For
  instance, for $\mu = (1)$ and $\lambda = (n-1, 1)$, there are $n-1$ standard tableaux with
  major indexes $1, \ldots, n-1$ (alternatively, \eqref{eq:hook_q} results in
  $q[n-1]_q$). The remaining cases are omitted.
\qed\end{proof}

\begin{proof}[of Theorem \ref{thm:dist}]
  If $f^\lambda \geq n^5$, then \eqref{eq:fl_bound} gives
  \begin{equation}
    \label{eq:char_ratio3}
    \ln \frac{|\chi^\lambda(\ell^s)|}{f^\lambda}
      \leq \left(1 - \frac{1}{\ell}\right) \left[-\frac{9}{2}\ln n + \ln \sqrt{2\pi}\right]
             + \frac{\ell}{12n} - \frac{1}{2} \ln \ell
  \end{equation}
  As before one can check that the right-hand side of \eqref{eq:char_ratio3} is less than
  $\ln \frac{1}{\phi(\ell)n^2}$ for $\ell=2, 3$ and $n \geq 3$. When $\ell \geq 4$,
  term-by-term estimates give
  \begin{equation*}
      \ln \frac{|\chi^\lambda(\ell^s)|}{f^\lambda} \leq
         -\frac{9}{2} \left(1 - \frac{1}{4}\right) \ln n
         + \ln \sqrt{2\pi} + \frac{1}{12} - \frac{1}{2} \ln 4
  \end{equation*}
  which is less than $\ln \frac{1}{n^3}$ for $n \geq 3$. The first part of
  Theorem \ref{thm:dist} now follows from Lemma \ref{lem:phi_d} with $d=2$ for $n \geq 3$.
  It remains true for $n=1, 2$.
  
  For the second part, suppose $n \geq 81$, $\lambda_1 < n-7$, and $\lambda_1' < n-7$.
  It follows from Proposition \ref{prop:h_op} that $N$ from \eqref{eq:N} satisfies
  $N \geq 8$. Hence by Corollary \ref{cor:h_op} we have
    \[ f^\lambda \geq \frac{1}{9} \binom{n}{8} \geq n^5. \]
\qed\end{proof}

\section{Alternative Proof of the Hook Formula}
\label{sec:hook}

The proof of Theorem \ref{thm:hook} in \cite{fl95} and \cite{jk81} uses a certain
decomposition of the $r$-rim hook partition lattice and the original hook length formula.
We present an alternative proof following a different tradition,
instead generalizing the approach to the original hook length formula in
\cite[Corollary 7.21.6]{ec2}. A by-product of our proof is a particularly explicit description of
the movement of hook lengths mod $\ell$ as length $\ell$ ribbons are added to a partition
shape.

We are not at present aware of any other proofs or direct uses of Theorem \ref{thm:hook}, and
it seems to have been neglected by the literature. Indeed, the author empirically rediscovered it
and found the following proof before unearthing \cite{fl95}.

\begin{proof}[of Theorem \ref{thm:hook}]

Let $\lambda \vdash n$, $n = \ell s$. If $\lambda$ cannot be written as $s$ successive
ribbons of length $\ell$, then by the classical Murnaghan-Nakayama rule
\cite[Eq.~(7.75)]{ec2} we have $\chi^\lambda(\ell^s) = 0$, so assume $\lambda$ can be so
written.

Combining \eqref{eq:cyc_exps}, \eqref{eq:hook_q}, and \eqref{eq:lin_comb1}
shows that we may compute $\chi^\lambda(\ell^s)$ by letting $q \to \omega_n^s$ in
the right-hand side of \eqref{eq:hook_q}. We may replace each $q$-number
$[a]_q$ with $q^a - 1$ by canceling the $q - 1$'s, since $\lambda \vdash n$. Since
$\omega_n^s$ has order $\ell$, the values of $q^a - 1$ at $\omega_n^s$ depend only
on $a$ mod $\ell$. Moreover, $q^a - 1$ has only simple roots, and it has a root at
$\omega_n^s$ if and only if $\ell \mid a$. The order of vanishing of the numerator at
$q = \omega_n^s$ is then $\#\{i \in [n] : i \equiv_\ell 0\} = s$, and the order of
vanishing of the denominator is $\#\{c \in \lambda : h_c \equiv_\ell 0\}$. The
following lemma ensures these counts agree. We postpone the proof to the end of
this section.

\begin{lemma}
  \label{lem:fiberpairs}
  Let $\lambda \vdash n$, $n = \ell s$, and suppose $\lambda$
  can be written as a sequence of $s$ successive ribbons of length $\ell$.
  Then for any $a \in \bZ$,
    \[ \#\{c \in \lambda : h_c \equiv_{\ell} \pm a\}
        = s \cdot \#\{a, -a \text{ (mod $\ell$)}\}. \]
\end{lemma}

\noindent Here $\#\{a, -a \text{ (mod $\ell$)}\}$ is $1$ if $a \equiv_\ell -a$ and
$2$ otherwise.

We may now compute the desired $q \to \omega_n^s$ limit by repeated
applications of L'Hopital's rule. In particular, we find
\begin{equation}
  \label{eq:rect_char}
  |\chi^\lambda(\ell^s)|
      = \left|\lim_{q \to \omega_n^s} q^{d(\lambda)}
         \frac{\prod_{i \in [n]} [i]_q}{\prod_{c \in \lambda} [h_c]_q}\right|
      = \left|\lim_{q \to \omega_n^s}
         \frac{\prod\limits_{\substack{i \in [n] \\ i \not\equiv_\ell 0}} q^i - 1}
                 {\prod\limits_{\substack{c \in \lambda \\ h_c \not\equiv_\ell 0}} q^{h_c} - 1}
         \right|
         \left|
         \frac{\prod\limits_{\substack{i \in [n] \\ i \equiv_\ell 0}} i\omega_n^{s(i-1)}}
                 {\prod\limits_{\substack{c \in \lambda \\ h_c \equiv_\ell 0}}
                       h_c\omega_n^{s(h_c-1)}}
         \right|
\end{equation}

The second factor in the right-hand side of \eqref{eq:rect_char} equals the
right-hand side of \eqref{eq:hook_prod}, so we must show the first factor in the
right-hand side of \eqref{eq:rect_char} is $1$. For that, note that $q^a - 1$ at
$q = \omega_n^s$ for $a \not\equiv_\ell 0$ is non-zero and is conjugate to
$q^{-a} - 1$ at $q = \omega_n^s$. By Lemma \ref{lem:fiberpairs}, it follows 
that the contribution to the overall magnitude due to $\{c \in \lambda : h_c \equiv_{\ell} a
\text{ or } -a\}$ cancels with the contribution due to  $\{i \in [n] : i \equiv_{\ell} a \text{ or }
-a\}$ for each $a \not\equiv_\ell 0$. This completes the proof of the theorem.

\qed\end{proof}

As for Lemma \ref{lem:fiberpairs}, it is an immediate consequence of the following somewhat
more general result.

\begin{lemma}
  \label{lem:fiberdiffs}
  Suppose $\lambda/\mu$ is a ribbon of length $\ell$. For any $a \in \bZ$,
    \[ \#\{c \in \mu : h_c \equiv_\ell \pm a\} + \#\{a, -a \text{ (mod $\ell$)}\}
        = \#\{d \in \lambda : h_d \equiv_\ell \pm a\}. \]
\end{lemma}
  \begin{proof}
	We determine how the counts $\#\{c \in \mu : h_c \equiv_\ell \pm a\}$ change when
	adding a ribbon of length $\ell$; see Figure \ref{fig:ribbon_regions}. We define the following
	regions in $\lambda$, relying on French notation to determine the meaning of ``leftmost,''
	etc.
    \begin{enumerate}[(I)]
      \item Cells $c \in \mu$ where $c$ is not in the same row or column as
        any element of $\lambda/\mu$.
      \item Cells $c \in \mu$ which are in the same row as some element of
        $\lambda/\mu$ and are strictly left of the leftmost cell in $\lambda/\mu$.
      \item Cells $c \in \mu$ which are in the same column as some element of
        $\lambda/\mu$ and are strictly below the bottommost cell of $\lambda/\mu$.
      \item Cells $c \in \lambda$ which are in both the same column and row as some
        element(s) of $\lambda/\mu$. Region (IV) includes the ribbon
        $\lambda/\mu$ itself.
    \end{enumerate}

    \begin{figure}[ht]
      \centering
      \includegraphics[scale=1]{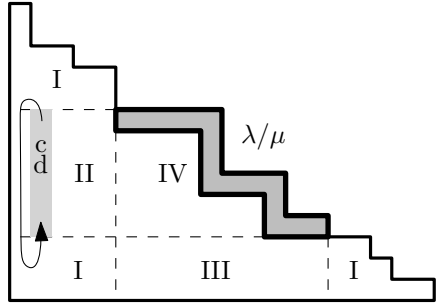}
      \caption{All regions of a partition $\lambda$ where $\lambda/\mu$ is a ribbon}
      \label{fig:ribbon_regions}
    \end{figure}

    \begin{figure}[ht]
      \centering
      \includegraphics[scale=1]{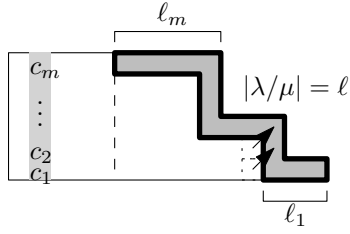}
      \caption{Regions (II) and (IV) up close}
      \label{fig:ribbon_regionII}
    \end{figure}

    We now describe how hook lengths change in each region, mod the ribbon length $\ell$,
    in going from $\mu$ to $\lambda$. They are unchanged in region (I). Regions (II) and (III) are
    similar, so we consider region (II). This region is a rectangle, which we imagine breaking up
    into columns. Write $h_c^\lambda$ or $h_c^\mu$ to denote the hook length of a cell
    $c \in \mu$ as an element of $\lambda$ or $\mu$, respectively. For $c$ in region (II), let
    $d$ denote the cell in region (II) immediately below $c$, with wrap-around. We claim
    $h_c^\lambda \equiv_\ell h_d^\mu$. Given the claim, hook lengths mod $\ell$ in regions
    (II) and (III) are simply permuted in going from $\mu$ to $\lambda$, so changes to the
    counts $\#\{c \in \mu : h_c^\mu \equiv_\ell \pm a\}$ arise only from region (IV).

    For the claim, let $c_1, c_2, \ldots, c_m$ be the cells of the column in region (II) containing
    $c$, listed from bottom to top; see Figure \ref{fig:ribbon_regionII}. Begin by comparing hook
    lengths at $c_1$ and $c_2$. Since $\lambda - \mu$ is a ribbon, the rightmost cell of
    $\mu$ in the same row as $c_1$ is directly left and below the rightmost cell of $\lambda$
    in the same row as $c_2$. It follows that $h_{c_1}^\mu = h_{c_2}^\lambda$.
    This procedure yields the claim except when $c = c_1$. In that case, $d = c_m$, and we 
    further claim $h_{c_1}^\lambda = h_{c_m}^\mu + \ell$, which will finish the argument. 
    Indeed, let $\ell_i$ denote the number of elements in $\lambda - \mu$ in the same row as
    $c_i$. Certainly $\ell = \ell_1 + \cdots + \ell_m$. Further,
    $h_{c_i}^\lambda = h_{c_i}^\mu + \ell_i$. Putting it all together, we have
    \begin{align*}
      h_{c_1}^\lambda
        &= h_{c_1}^\mu + \ell_1
          = h_{c_2}^\lambda + \ell_1 \\
        &= h_{c_2}^\mu + \ell_2 + \ell_1
          = \cdots \\
        &= h_{c_m}^\mu + \ell_m + \cdots + \ell_2 + \ell_1
          = h_{c_m}^\mu + \ell.
    \end{align*}

    We now turn to region (IV). It suffices to consider the case depicted in
    Figure \ref{fig:ribbon_regionIV}, where regions (I), (II), and (III) are empty. We define two
    more regions as follows; see Figure \ref{fig:ribbon_regionIV}.
    \begin{enumerate}[(A)]
      \item Cells $c \in \lambda$ in the first row or column.
      \item Cells $c \in \lambda$ not in the first row or column.
    \end{enumerate}

    \begin{figure}[ht]
      \centering
      \includegraphics[scale=1]{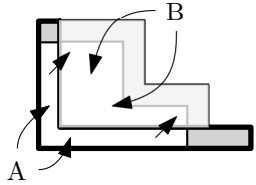}
      \caption{Regions (A) and (B) of a partition $\mu$ where $\lambda/\mu$ is a ribbon}
      \label{fig:ribbon_regionIV}
    \end{figure}
    \begin{figure}[ht]
      \centering
      \includegraphics[scale=1]{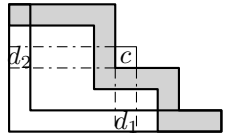}
      \caption{Adding a cell to region (B)}
      \label{fig:ribbon_regionIVb}
    \end{figure}

    Region (B) is precisely $\mu$ translated up and right one square. Moreover,
    this operation preserves hook lengths, so changes in the counts
    $\#\{c \in \mu : h_c^\mu \equiv_\ell \pm a\}$ arise entirely from region (A).
    We have thus reduced the lemma to the statement
    \begin{equation}
      \label{eq:region_claim}
      \#\{c \text{ in region (A) } : h_c^\lambda \equiv_\ell \pm a\}
        = \#\{a, -a \text{ (mod $\ell$)}\}.
    \end{equation}
    We prove \eqref{eq:region_claim} by induction on the size of region (B). In the base case, 
    region (B) is empty, so $\lambda$ is a hook, and the result is easy to see directly (for
    instance, negate the hook lengths in only the ``vertical leg'' to get entries of precisely
    $1, 2, \ldots, \ell$). For the inductive step, consider the effect of adding a cell $c$ to region
    (B). Now $c$ is in the same column as some cell $d_1$ in region (A) and $c$ is in the same
    row as some cell $d_2$ in region (A); see Figure \ref{fig:ribbon_regionIVb}. Say the original
    hook length of $d_1$ is $i$ and the original hook length of $d_2$ is $j$. It is easy to see
    that $i+j = \ell - 1$. Adding $c$ to region (B) increases the hook lengths $i$ and $j$ each
    by $1$, but $j+1 \equiv_\ell -i$ and $i+1 \equiv_\ell -j$, so the required counts remain as
    claimed in the inductive step. This completes the proof of the lemma and, hence,
    Theorem \ref{thm:hook}.
  \qed\end{proof}

We briefly contrast our approach with that of \cite{fl95}. Let $f^\lambda_\ell$ be the number
of ways to write $\lambda$ as successive ribbons each of length $\ell$. If $\lambda \vdash n =
\ell s$, by the Murnaghan-Nakayama rule $\chi^\lambda(\ell^s)$ is a signed sum over terms
counted by $f^\lambda_\ell$. While there is typically cancellation in this sum, there is in fact
none for rectangular cycle types \cite[2.7.26]{jk81}, i.e.~$\chi^\lambda(\ell^s) =
\pm f^\lambda_\ell$. Indeed, \cite{fl95} proved Theorem \ref{thm:hook} using standard rim
hook tableaux instead of character evaluations, though virtually every application of their result 
uses the character-theoretic inequality in Theorem \ref{thm:fl_bound}.

The sign of $\chi^\lambda(\ell^s)$ can be computed in terms of \textit{abaci} as in
\cite[2.7.23]{jk81}. The sign may also be computed ``greedily'' by repeatedly removing
$\ell$-rim hooks from $\lambda$ in any order whatsoever, which is a consequence of
(among other things) the following corollary of Lemma \ref{lem:fiberdiffs} and Theorem
\ref{thm:hook}. We have been unable to find part (iv) in the literature, though for the rest see
\cite[2.5-2.7]{fl95} and their references.

\begin{corollary}\label{cor:ell_decomposable}
  Let $\lambda \vdash n = \ell s$. The following are equivalent:
  \begin{enumerate}[(i)]
    \item $\chi^\lambda(\ell^s) \neq 0$;
    \item $\lambda$ can be written as successive length $\ell$ ribbons, i.e.~the $\ell$-core
      of $\lambda$ is empty;
    \item we have
      \[ \#\{c \in \lambda : h_c \equiv_\ell 0\} = s; \]
    \item for any $a \in \bZ$,
      \[ \#\{c \in \lambda : h_c \equiv_{\ell} \pm a\}
          = s \cdot \#\{a, -a \text{ (mod $\ell$)}\}. \]
  \end{enumerate}
\end{corollary}
  \begin{proof}
    (i) and (ii) are equivalent by Theorem \ref{thm:hook}. (ii) implies (iv) by Lemma
    \ref{lem:fiberpairs} and (iv) implies (iii) trivially. Finally, (iii) is equivalent to (i) as follows. The
    expression \eqref{eq:hook_q} is a polynomial, so the order of vanishing at $q \to
    \omega_n^s$ of the numerator, namely $s$, is at most as large as the order of vanishing of
    the denominator, namely $\#\{c \in \lambda : h_c \equiv_\ell 0\}$. The limiting ratio is
    non-zero if and only if these counts agree, so (iii) is equivalent to (i).
  \qed\end{proof}

While Corollary \ref{cor:ell_decomposable}
gives equivalent conditions for $\chi^\lambda(\ell^s) \neq 0$, \cite[Corollary 7.5]{MR768993}
gives interesting and different necessary conditions for $\chi^\lambda(\nu) \neq 0$ for general
shapes $\nu$.

\section{Unimodality and $\chi^\lambda(\mu)$}
\label{sec:unimodality}

We end with a brief discussion of inequalities related to symmetric group characters.
In applying Proposition \ref{prop:h_op1}, we essentially replaced
$\frac{n!}{\prod_{c \in \lambda} h_c}$ with $\frac{n!}{\prod_{c \in \lambda} h_c^{\op}}$,
since the latter is order-reversing with respect to the diagonal preorder by
\eqref{eq:diag_preorder}. Moreover, it is relatively straightforward to mutate partitions and
predictably increase or decrease them in the diagonal preorder, as in the proof of
Proposition \ref{prop:h_op}. It would be desirable to instead work directly with symmetric group
characters themselves and appeal to general results about how $|\chi^\lambda(\mu)|$
increases or decreases as $\lambda$ is mutated and $\mu$ is held fixed, though we have
found very few concrete and no conjectural results in this direction. Any progress seems both
highly non-trivial and potentially useful, so in this section we record some initial observations.

We have $\chi^{(a+1, 1^b)}(1^n) = \binom{n-1}{a}$ for $a+b+1=n$, so these values are
unimodal in $a$. Using Theorem \ref{thm:hook} shows more generally
that for all $\ell \mid n$,
  \[ |\chi^{(a+1, 1^b)}(\ell^{n/\ell})|
       = \binom{\frac{n}{\ell} - 1}
                     {\left\lfloor \frac{a}{\ell}\right\rfloor} \]
which is again unimodal in $a$. However, $|\chi^\lambda(\ell^s)|$ does not seem to respect 
changes in $\lambda$ under dominance order in general in any suitable sense. On the other 
hand, if we allow the cycle type $\mu$ to vary and consider the Kostka numbers
$K_{\lambda\mu}$ as a surrogate for $|\chi^\lambda(\mu)|$ (since $K_{\lambda(1^n)} =
\chi^\lambda(1^n)$), we have a series of well-known and very general inequalities. We write
$K_{\lambda\mu}(t)$ for the Kostka-Foulkes polynomial and $\nu \geq \mu$ for dominance
order. We have:

\begin{theorem}[{\cite{snapper71}, \cite{lv73}, \cite{lam77}; \cite{gp92}}]
  $K_{\lambda\nu} \leq K_{\lambda\mu}$ for all $\lambda$ if and only if
  $\nu \geq \mu$. Indeed, $\nu \geq \mu$ implies $K_{\lambda\nu}(t) \leq
  K_{\lambda\mu}(t)$ (coefficient-wise) for all $\lambda$.
\end{theorem}

\begin{question}
  Are there any ``nice'' infinite families besides hooks and rectangles for which
  $|\chi^\lambda(\mu)|$ is monotonic, unimodal, or suitably order-preserving as $\lambda$
  varies? What about as $\mu$ varies?
\end{question}

\section{Acknowledgements}
  The author would like to thank Sheila Sundaram for sharing a preprint of \cite{sundaram16}
  which motivated the present work. He would also like to thank his advisor, Sara Billey, for her
  support, insightful comments, and a careful reading of the manuscript; his partner, R.~Andrew
  Ohana, for numerous fruitful discussions and support, including an early observation which lead
  to Lemma \ref{lem:fiberpairs}; and Connor Ahlbach for valuable discussions on related work.

\bibliographystyle{alpha}
\bibliography{refs}{}

\begin{thebibliography}{CFKP11}

\bibitem[ABR05]{abr05}
Ron~M. Adin, Francesco Brenti, and Yuval Roichman.
\newblock Descent representations and multivariate statistics.
\newblock {\em Trans. Amer. Math. Soc.}, 357(8):3051--3082 (electronic), 2005.

\bibitem[CFKP11]{cfkp11}
Ionu{\c{t}} Ciocan-Fontanine, Matja{\v{z}} Konvalinka, and Igor Pak.
\newblock The weighted hook length formula.
\newblock {\em J. Combin. Theory Ser. A}, 118(6):1703--1717, 2011.

\bibitem[D{\'e}s90]{desarmenien90}
Jacques D{\'e}sarm{\'e}mien.
\newblock {{\'E}}tude modulo $n$ des statistiques mahoniennes.
\newblock {\em S{\'e}minaire Lotharingien de Combinatoire}, 22:27--35, 1990.

\bibitem[FL95]{fl95}
Sergey Fomin and Nathan Lulov.
\newblock On the number of rim hook tableaux.
\newblock {\em Zap. Nauchn. Sem. S.-Peterburg. Otdel. Mat. Inst. Steklov.
  (POMI)}, 223(Teor. Predstav. Din. Sistemy, Kombin. i Algoritm. Metody.
  I):219--226, 340, 1995.

\bibitem[Fou72]{foulkes72}
H.~O. Foulkes.
\newblock Characters of symmetric groups induced by characters of cyclic
  subgroups.
\newblock In {\em Combinatorics ({P}roc. {C}onf. {C}ombinatorial {M}ath.,
  {M}ath. {I}nst., {O}xford, 1972)}, pages 141--154. Inst. Math. Appl.,
  Southend-on-Sea, 1972.

\bibitem[Ful97]{fulton97}
W.~Fulton.
\newblock {\em Young Tableaux; with applications to representation theory and
  geometry}, volume~35 of {\em London Mathematical Society Student Texts}.
\newblock Cambridge University Press, New York, 1997.

\bibitem[GP92]{gp92}
A.~M. Garsia and C.~Procesi.
\newblock On certain graded {$S_n$}-modules and the {$q$}-{K}ostka polynomials.
\newblock {\em Adv. Math.}, 94(1):82--138, 1992.

\bibitem[JK81]{jk81}
Gordon James and Adalbert Kerber.
\newblock {\em The representation theory of the symmetric group}, volume~16 of
  {\em Encyclopedia of Mathematics and its Applications}.
\newblock Addison-Wesley Publishing Co., Reading, Mass., 1981.
\newblock With a foreword by P. M. Cohn, With an introduction by Gilbert de B.
  Robinson.

\bibitem[Joh07]{MR2276964}
Marianne Johnson.
\newblock Standard tableaux and {K}lyachko's theorem on {L}ie representations.
\newblock {\em J. Combin. Theory Ser. A}, 114(1):151--158, 2007.

\bibitem[Ker93]{kerov93}
S.~Kerov.
\newblock A {$q$}-analog of the hook walk algorithm for random {Y}oung
  tableaux.
\newblock {\em J. Algebraic Combin.}, 2(4):383--396, 1993.

\bibitem[Kly74]{klyachko74}
A.~A. Klyachko.
\newblock Lie elements in the tensor algebra.
\newblock {\em Siberian Mathematical Journal}, 15(6):914--920, 1974.

\bibitem[Kno75]{knopfmacher75}
John Knopfmacher.
\newblock {\em Abstract analytic number theory}.
\newblock North-Holland Publishing Co., Amsterdam-Oxford; American Elsevier
  Publishing Co., Inc., New York, 1975.
\newblock North-Holland Mathematical Library, Vol. 12.

\bibitem[KS06]{MR2228926}
L.~G. Kov\'acs and Ralph St\"ohr.
\newblock A combinatorial proof of {K}lyachko's theorem on {L}ie
  representations.
\newblock {\em J. Algebraic Combin.}, 23(3):225--230, 2006.

\bibitem[KW01]{kw01}
Witold Kra{\'s}kiewicz and Jerzy Weyman.
\newblock Algebra of coinvariants and the action of a {C}oxeter element.
\newblock {\em Bayreuth. Math. Schr.}, (63):265--284, 2001.

\bibitem[Lam78]{lam77}
T.~Y. Lam.
\newblock Young diagrams, {S}chur functions, the {G}ale-{R}yser theorem and a
  conjecture of {S}napper.
\newblock {\em J. Pure Appl. Algebra}, 10(1):81--94, 1977/78.

\bibitem[LS08]{ls08}
Michael Larsen and Aner Shalev.
\newblock Characters of symmetric groups: sharp bounds and applications.
\newblock {\em Invent. Math.}, 174(3):645--687, 2008.

\bibitem[LV73]{lv73}
R.~A. Liebler and M.~R. Vitale.
\newblock Ordering the partition characters of the symmetric group.
\newblock {\em J. Algebra}, 25:487--489, 1973.

\bibitem[MPP17]{mpp17}
Alejandro Morales, Igor Pak, and Greta Panova.
\newblock Asymptotics of the number of standard {Y}oung tableaux of skew shape.
\newblock Preprint, March 2017.

\bibitem[Pak]{243846}
Igor Pak.
\newblock Inequality for hook numbers in young diagrams.
\newblock MathOverflow.
\newblock URL:https://mathoverflow.net/q/243846 (version: 2017-04-13).

\bibitem[Reu93]{reutenauer93}
Christophe Reutenauer.
\newblock {\em Free {L}ie algebras}, volume~7 of {\em London Mathematical
  Society Monographs. New Series}.
\newblock The Clarendon Press, Oxford University Press, New York, 1993.
\newblock Oxford Science Publications.

\bibitem[Roi96]{roichman96}
Yuval Roichman.
\newblock Upper bound on the characters of the symmetric groups.
\newblock {\em Invent. Math.}, 125(3):451--485, 1996.

\bibitem[Sag01]{sagan01}
Bruce~E. Sagan.
\newblock {\em The symmetric group}, volume 203 of {\em Graduate Texts in
  Mathematics}.
\newblock Springer-Verlag, New York, second edition, 2001.
\newblock Representations, combinatorial algorithms, and symmetric functions.

\bibitem[Sch03]{schocker03}
Manfred Schocker.
\newblock Embeddings of higher {L}ie modules.
\newblock {\em J. Pure Appl. Algebra}, 185(1-3):279--288, 2003.

\bibitem[Ser77]{serre77}
Jean-Pierre Serre.
\newblock {\em Linear representations of finite groups}.
\newblock Springer-Verlag, New York-Heidelberg, 1977.
\newblock Translated from the second French edition by Leonard L. Scott,
  Graduate Texts in Mathematics, Vol. 42.

\bibitem[Sna71]{snapper71}
Ernst Snapper.
\newblock Group characters and nonnegative integral matrices.
\newblock {\em J. Algebra}, 19:520--535, 1971.

\bibitem[Spe14]{spencer14}
Joel Spencer.
\newblock {\em Asymptopia}, volume~71 of {\em Student Mathematical Library}.
\newblock American Mathematical Society, Providence, RI, 2014.
\newblock With Laura Florescu.

\bibitem[Spr74]{springer74}
T.~A. Springer.
\newblock Regular elements of finite reflection groups.
\newblock {\em Invent. Math.}, 25:159--198, 1974.

\bibitem[Sta79]{stanley79}
Richard~P. Stanley.
\newblock Invariants of finite groups and their applications to combinatorics.
\newblock {\em Bull. Amer. Math. Soc. (N.S.)}, 1(3):475--511, 1979.

\bibitem[Sta84]{MR768993}
Richard~P. Stanley.
\newblock The stable behavior of some characters of {${\rm SL}(n,{\bf C})$}.
\newblock {\em Linear and Multilinear Algebra}, 16(1-4):3--27, 1984.

\bibitem[Sta99]{ec2}
R.~P. Stanley.
\newblock {\em Enumerative combinatorics. {V}ol. 2}, volume~62 of {\em
  Cambridge Studies in Advanced Mathematics}.
\newblock Cambridge University Press, Cambridge, 1999.
\newblock With a foreword by Gian-Carlo Rota and appendix 1 by Sergey Fomin.

\bibitem[Ste89]{stembridge89}
John~R. Stembridge.
\newblock On the eigenvalues of representations of reflection groups and wreath
  products.
\newblock {\em Pacific J. Math.}, 140(2):353--396, 1989.

\bibitem[Sun17]{sundaram16}
Sheila Sundaram.
\newblock On {C}onjugacy {C}lasses of {$S_n$} {C}ontaining all {I}rreducibles.
\newblock Accepted to \textit{Israel J. Math.}, February 2017.

\end{thebibliography}


\end{document}